\documentclass[12pt]{amsart}
\usepackage{amsmath}
\usepackage{amsfonts}
\usepackage{amssymb}

\newcommand{\Bignorm}[1]{\Bigl\|#1\Bigr\|}
\newcommand{\bignorm}[1]{\bigl\lVert#1\bigr\rVert}

\newcommand{\bC}{{\mathbb C}}
\newcommand{\Rst}{{\mathbb R}}
\newcommand{\RR}{{\mathbb R}}

\newcommand{\Rdst}{{\Rst^d}}

\newcommand{\Rtdst}{{\Rst^{2d}}}
\newcommand{\set}[2]{\big\{ \, #1 \, \big| \, #2 \, \big\}}
\newcommand{\norm}[1]{\lVert#1\rVert}
\newcommand{\Esp}{{E}}

\newcommand{\Zst}{{\mathbb Z}}
\newcommand{\Zdst}{{\Zst^d}}

\newcommand{\Lsp}{{L}}
\newcommand{\Ltsp}{{\Lsp^2}}
\newcommand{\LtRd}{{\Ltsp(\Rst^d)}}
\newcommand{\Lpsp}{{\Lsp^p}}
\newcommand{\LpRd}{{\Lpsp(\Rst^d)}}

\def\into{\hookrightarrow}

\newcommand{\ip}[2]{\ensuremath{\left<#1,#2\right>}}
\newcommand{\sett}[1]{\ensuremath{\left \{ #1 \right \}}}
\newcommand{\abs}[1]{\ensuremath{\left| #1 \right| }}
\newcommand{\Linf}{{L^\infty}}
\newcommand{\LinfRd}{{\Linf(\Rdst)}}

\newcommand{\Qb}{{[0,1/\beta)^d}}
\newcommand{\LpQb}{{L^p(\Qb)}}
\newcommand{\Lp}{L^p}
\newcommand{\LRd}{{L^1(\Rdst)}}

\newcommand{\latt}{\ensuremath{{\alpha \Zdst \times \beta \Zdst}}}
\newcommand{\latti}{\ensuremath{{\alpha_i \Zdst \times \beta_i \Zdst}}}
\newcommand{\winpqv}{{W(L^p,L^q_v)}}
\newcommand{\winpqvdual}{{W(L^{p'},L^{q'}_{1/v})}}
\newcommand{\winw}{{W(L^\infty,L^1_w)}}
\newcommand{\winwz}{{W(C_0,L^1_w)}}
\newcommand{\koethetoppqv}{\sigma(\winpqv, \winpqvdual)}
\newcommand{\wininfqvdual}{{W(L^{1},L^{q'}_{1/v})}}

\newcommand{\opAw}{{\mathcal{A}_w}}
\newcommand{\opAwE}{{\mathcal{A}_w}(E)}
\newcommand{\apw}{AP_w^p(\rep)}
\newcommand{\coef}[1]{{C_{#1}}}
\newcommand{\rec}[1]{{R_{#1}}}
\newcommand{\frameop}[1]{{S_{#1}}}
\newcommand{\coefgl}{\coef{g,\Lambda}}
\newcommand{\coefgili}{\coef{g^i,\Lambda^i}}
\newcommand{\recgl}{\rec{g,\Lambda}}
\newcommand{\rechl}{\rec{h,\Lambda}}
\newcommand{\recgili}{\rec{g^i,\Lambda^i}}

\newcommand{\seqpqv}{{S^{p,q}_v}}
\newcommand{\seqinfqv}{{S^{+\infty,q}_v}}
\newcommand{\seqpqvl}{{\seqpqv(\Lambda)}}
\newcommand{\seqpqvli}{{\seqpqv(\Lambda^i)}}
\newcommand{\seqpqvlo}{{\seqpqv(\Lambda^1)}}
\newcommand{\seqpqvln}{{\seqpqv(\Lambda^n)}}

\newcommand{\gab}{{\mathcal{G}}}
\newcommand{\opgab}{\frameop{\gab}}
\newcommand{\opgabi}{\frameop{g^i,\Lambda^i}}
\newcommand{\opgabo}{\frameop{g^1,\Lambda^1}}
\newcommand{\opgabn}{\frameop{g^n,\Lambda^n}}
\newcommand{\opgabinv}{\frameop{\gab}^{-1}}
\newcommand{\coefgab}{\coef{\gab}}
\newcommand{\recgab}{\rec{\gab}}

\newcommand{\Fou}{{\mathcal{F}}}
\newcommand{\Mou}{{\mathcal{M}}}
\newcommand{\Nou}{{\mathcal{N}}}
\newcommand{\consv}{C_v}

\newcommand{\gabpartsums}{{R}}
\newcommand{\gabregpartsums}{{\sigma}}

\newcommand{\tilg}{\widetilde{g}}
\newcommand{\gili}{g^i_{\lambda^i}}
\newcommand{\tilgili}{\tilde{g}^i_{\lambda^i}}
\newcommand{\gilikj}{g^i_{(\alpha_i k,\beta_i j)}}
\newcommand{\tilgilikj}{\tilde{g}^i_{(\alpha_i k,\beta_i j)}}

\newcommand{\regwegjM}{r_{j,M}}
\newcommand{\regwegbijM}{r_{\beta_i j,M}}
\newcommand{\BLtRd}{{B(\LtRd)}}

\newcommand{\lwLinfRd}{{\ell^1_w(\Rdst,\LinfRd)}}

\newcommand{\regM}{\sigma_M}

\newcommand{\rep}{\rho}
\newcommand{\eps}{{\varepsilon}}
\newcommand{\Rdc}{{\hat{R}_c^d}}

\newtheorem{theo}{Theorem}
\newtheorem{lemma}{Lemma}
\newtheorem{coro}{Corollary}
\newtheorem{prop}{Proposition}
\newtheorem{remark}{Remark}

\title{Multi-window Gabor frames in amalgam spaces}



\author[R.~Balan]{Radu~Balan}
\address{Department of Mathematics \\
University of Maryland\\
College Park, MD 20742, USA}
\email[Radu~Balan]{rvbalan@math.umd.edu}

\author[J.~Christensen]{Jens~G.~Christensen} 
\address{Department of Mathematics \\
Colgate University\\
Hamilton, NY 13346, USA}
\email[Jens~G.~Christensen]{jchristensen@colgate.edu}

\author[I.~Krishtal]{Ilya~A.~Krishtal}
\address{Department of Mathematical Sciences \\
Northern Illinois University \\
DeKalb, IL 60115, USA}
\email[Ilya~A.~Krishtal]{krishtal@math.niu.edu}

\author[K.~Okoudjou]{Kasso~A.~Okoudjou}
\address{Department of Mathematics \\
University of Maryland\\
College Park, MD 20742, USA}
\email[Kasso~A.~Okoudjou]{kasso@math.umd.edu}

\author[J.~L.~Romero]{Jos\'e Luis Romero}
\address{Faculty of Mathematics, University of Vienna \\ Oskar-Morgenstern-Platz 1, A-1090 Wien, Austria}
\email[Jos\'e Luis Romero]{jose.luis.romero@univie.ac.at}

\begin{document}
\begin{abstract}
We show that  multi-window Gabor frames with windows in the Wiener algebra
$W(L^{\infty}, \ell^{1})$ are Banach frames for all Wiener amalgam spaces. As a by-product of our results we positively
answer an open question that was posed by  [Krishtal and Okoudjou, Invertibility of the Gabor frame
operator on the Wiener amalgam space, {\em {J}. {A}pprox. {T}heory},
153(2), 2008] and  concerns the continuity of the canonical dual of a Gabor frame with a continuous generator in the
Wiener algebra.  The proofs are based on a recent version of Wiener's $1/f$ lemma.
\end{abstract}
\maketitle

\section{Introduction}
\noindent A Gabor system is a collection of functions
$\gab(g, \Lambda) = \set{\pi(\lambda)g}{\lambda \in \Lambda}$,
where $\Lambda=\alpha\Zdst\times\beta\Zdst$ is a lattice, $g \in \LtRd$,
and the \emph{time-frequency shifts} of $g$ are given by
\begin{align*}
\pi(x,\omega) g (y) = e^{2 \pi i \omega \cdot y} g(y-x),
\qquad (y \in \Rdst).
\end{align*}
This system is called a \emph{frame} if
$\norm{f}_2^2 \approx \sum_\lambda \abs{\ip{f}{\pi(\lambda)g}}^2$.
In this case, there exists a \emph{dual Gabor system}
$\gab(\tilg, \Lambda) = \set{\pi(\lambda)\tilg}{\lambda \in \Lambda}$ providing the $L^2$-expansions
\begin{align}
\label{eq_expension}
f=\sum_\lambda \ip{f}{\pi(\lambda)g} \pi(\lambda) \tilg = \sum_\lambda \ip{f}{\pi(\lambda)\tilg} \pi(\lambda)g.
\end{align}
It is known that under suitable assumptions on $g$ and $\tilg$ that expansion extends to $L^p$ 
spaces \cite{badau03, grle01,grhe01,grheok02}. To some extent, these results parallel the theory of Gabor
expansions on modulation spaces \cite{fegr97}. However, since modulation spaces are defined in terms of time-frequency
concentration - and are indeed characterized by the \emph{size} of the numbers $\ip{f}{\pi(\lambda)g}$ - Gabor
expansions are also available in a more irregular context, where $\Lambda$ does not need to be a lattice. In
contrast, the theory of Gabor expansions in $L^p$ spaces relies on the strict algebraic structure of $\Lambda$. Indeed,
as shown in \cite{wa92}, Poisson summation formula implies that the frame operator $Sf := \sum_\lambda
\ip{f}{\pi(\lambda)g} \pi(\lambda)g$ can be written as 
\begin{align}
\label{eq_frame_op}
Sf(x)=\frac{1}{\beta^d} \sum_{j\in\Zdst} \sum_{k\in\Zdst}
\left(\overline{g(x-j/\beta-\alpha k)} g(x-\alpha k)\right) f(x-j/\beta).
\end{align}
This expression allows one to transfer spatial information about $g$ to boundedness properties of $S$ and is at the
core of the $L^p$-theory of Gabor expansions.

One often has explicit information only about $g$, while the existence of $\tilg$ is merely inferred from the frame
inequality. It is then important to know whether certain good properties of $g$ are also inherited by $\tilg$,
so as to deduce the validity of \eqref{eq_expension} in various function spaces. The key technical point is showing
the $S$ in invertible not only in $L^2$ but also in the other relevant spaces. This was proved for modulation
spaces in \cite{grle04, gr04} and for $L^p$ spaces in \cite{krok08}. In this latter case the analysis relies
on the fact that $S^{-1}$ is the frame operator associated with the dual Gabor system
$\gab(\tilg, \Lambda)$ and thus admits an expansion like the one in \eqref{eq_frame_op}.

The objective of this article is to extend the $L^p$-theory of Gabor expansions to multi-window Gabor systems
(see \cite{ba99,hala00}),
\begin{align*}
\gab(\Lambda^1, \ldots, \Lambda^n, g^1, \ldots, g^n)
= \set{\pi(\lambda^i)g^i}{\lambda^i \in \Lambda^i, 1 \leq i \leq n},
\end{align*}
where $\Lambda^1, \ldots, \Lambda^n \subseteq \Rtdst$ are lattices
$\Lambda^i=\latti$ and $g^1, \ldots, g^n: \Rdst \to \bC$. The challenge in doing so is that,
in contrast to the case of a single lattice $\Lambda$, the corresponding dual system
does not consist of lattice time-frequency translates of a certain family of functions
$\widetilde{g^1}, \ldots, \widetilde{g^n}$. The main technical point of this article is 
to show that, nevertheless, $S^{-1}$ admits a generalized expansion
\begin{align}
\label{eq_frame_op_1}
S^{-1}f(x)=\sum_{k} G_k(x) f(x-x_k),
\end{align}
where now the family of points $\sett{x_k}_k$ may not be contained in a lattice. We then prove that 
certain spatial localization properties of $g^1,\ldots, g^n$ imply corresponding
localization properties for the family
$\sett{G_k}_k$, and deduce that 
$S^{-1}$ is bounded on $L^p$-spaces. For technical reasons we work in the more general context of Wiener amalgam
spaces, that are spaces of functions that belong locally to $L^q$ and globally to $L^p$.

To achieve this, we study a Banach algebra of operators admitting an expansion like in \eqref{eq_frame_op_1}
with a suitable summability condition. We then resort to a recent Wiener-type result on non-commutative almost-periodic
Fourier series \cite{ba10-1} to prove that this algebra is spectral within the class of bounded operators on $L^p$.
This means that if an operator from that algebra is invertible on $L^p$, then the inverse operator necessarily
belongs to the algebra.
This approach is now common in time-frequency analysis \cite{st01-1, abk08, ba10-1, ba06-1, bacahela06, bacahela06-1,
BK11, fegr97, gr04,  grle04, ja95, K11} but its application to spaces that are not characterized by time-frequency
decay is rather subtle. As a by-product, we obtain consequences that are new even for the case of one generator. We
prove that if all the functions $g^i$ are continuous, so is every function in the dual system. This question was
posed in \cite{krok08}.

This paper is organized as follows. In Section~\ref{sec2} we define Wiener amalgam spaces and recall their
characterization via Gabor frames. In Section~\ref{sec3} we present the main technical result of this paper: a spectral
invariance theorem for a sub-algebra of weighted-shift operators in $B(L^{p}(\Rdst))$. In Section~\ref{sec4}, we use the
result of the previous section to extend the theory of multi-window Gabor frames to the class of Wiener amalgam spaces.
In particular, this last section contains a Wiener-type lemma for multi-window Gabor frames. 

\section{Amalgam spaces and Gabor expansions}\label{sec2} 
Before introducing the Wiener amalgam spaces, we first set the notation that will be used throughout the paper. 

Given $x,\omega \in \Rdst$, the translation and modulation operators act on a
function $f: \Rdst \to \bC$ by
\begin{align*}
T_x f (y) := f(y-x),
\qquad
M_\omega f (y) := e^{2 \pi i \omega \cdot y} f(y),
\end{align*}
where  $\omega \cdot y$ is the usual dot product. The time-frequency shift
associated with the point $\lambda=(x,\omega) \in \Rdst \times \Rdst$ is the operator
$\pi(\lambda)=\pi(x,\omega) := M_\omega T_x$.

Given two non-negative functions $f,g$, we write $f \lesssim g$ if $f \leq C
g$, for some constant $C>0$. If $\Esp$ is a Banach space, we denote by
$B(\Esp)$ the Banach algebra of all bounded linear operators on $\Esp$.

We use the following normalization of the Fourier transform of a function
$f: \Rdst \to \bC$:
\begin{align*}
\hat{f}(\omega) := \int_\Rdst f(x) e^{-2 \pi i \omega \cdot x} dx. 
\end{align*}

\subsection{Definition and properties of the amalgam spaces}\label{amalgamspacedef}
A function $w:\Rdst \to (0,+\infty)$  is called a {\em weight} if it is
continuous and symmetric (i.e. $w(x)=w(-x)$). A weight $w$ is
\emph{submultiplicative} if
\begin{align*}
w(x+y) \leq w(x)w(y),
\quad x,y \in \Rdst. 
\end{align*}
Prototypical examples are given by the polynomial weights
$w(x)=(1+\abs{x})^s$,
which are submultiplicative if $s \geq 0$. The main results in this article 
require to consider an extra condition on the weights. A weight $w$ is 
called \emph{admissible} if $w(0)=1$, it is submultiplicative  and satisfies 
the \emph{Gelfand-Raikov-Shilov} condition,
\begin{align*}
\lim_{k \rightarrow \infty} w(kx)^{1/k} = 1,
\quad x \in \Rdst.
\end{align*}
Note that this condition, together with the submultiplicativity, implies that
$w(x) \geq 1$, $x\in \Rdst$.

Given a submultiplicative weight $w$, a second weight $v:\Rdst \to
(0,+\infty)$ is called \emph{$w$-moderate} if there exists a constant
$\consv >0$ such that,
\begin{align}
\label{eq_consv}
v(x+y) \leq \consv w(x)v(y),
\quad x,y \in\Rdst.
\end{align}
For polynomial weights $v(x)=(1+\abs{x})^t, w(x)=(1+\abs{x})^s$, $v$ is
$w$-moderate if $\abs{t} \leq s$. If $v$ is $w$-moderate, it follows from~\eqref{eq_consv} and the symmetry of $w$ that $1/v$ is also $w$-moderate
(with the same constant).

Let $w$ be a submultiplicative weight and let $v$ be $w$-moderate. This will
be the standard assumption in this article. We will keep the weight $w$
fixed and consider classes of function spaces related to various weights $v$.
For $1 \leq p,q \leq +\infty$, we define the \emph{Wiener
amalgam space} $\winpqv$ as the class of all measurable functions
$f: \Rdst \to \bC$ such that,
\begin{align}
\label{eq_def_winpqv}
\norm{f}_\winpqv :=
\left(
\sum_{k \in \Zdst} \norm{f}^q_{\Lp([0,1)^d+k)} v(k)^q
\right)^{1/q} <\infty,
\end{align}
with the usual modifications when $q=+\infty$. As with Lebesgue spaces,
we identity two functions if they coincide almost everywhere. For a study of
this class of spaces in a much broader context see \cite{fe83, fe92-3, fost85}.
We only point out that, as a consequence of the assumptions on the weights $v$
and $w$, it can be shown that the partition $\{[0,1)^d+k: k \in \Zdst \}$ in~\eqref{eq_def_winpqv} can be replaced by more general coverings yielding an
equivalent norm.

Weighted amalgam spaces are \emph{solid}. This means that if $f \in \winpqv$
and $m \in \LinfRd$, then $mf \in \winpqv$ and
\begin{align}
\label{eq_solid_winpqv}
\norm{mf}_\winpqv \leq
\norm{m}_{\LinfRd} \norm{f}_\winpqv.
\end{align}
In addition, using the fact that $v$ is $w$-moderate, it follows that $\winpqv$
is closed under translations and
\begin{align}
\label{eq_trans_winpqv}
\norm{T_x f}_\winpqv \leq \consv w(x) \norm{f}_\winpqv,
\end{align}
where $\consv$ is the constant in~\eqref{eq_consv}.

The \emph{K\"othe-dual} of $\winpqv$ is the space of all
measurable functions $g:\Rdst \to \bC$ such that $g \cdot \winpqv \subseteq
\LRd$. It is equal to $\winpqvdual$, where $1/p+1/p'=1/q+1/q'=1$ for all  $1\leq p, q \leq  \infty$.
In particular, the pairing
\begin{align*}
\ip{\cdot}{\cdot}:\winpqv \times \winpqvdual \to \bC,
\qquad
\ip{f}{g} =\int_\Rdst f(x) \overline{g(x)}dx,
\end{align*}
is bounded.
The functionals arising from integration against functions in
$\winpqvdual$
determine a topology in $\winpqv$ that will be denoted by $\koethetoppqv$.

\subsection{Gabor expansions on amalgam spaces}
\label{sec_gab_lat}
We now recall the theory of Gabor expansions on Wiener amalgam
spaces as developed in \cite{fewe07, grle01, grhe01, grheok02}. Let $\Lambda = \latt$ be a
(separable) lattice which will be used to index time-frequency shifts.
For convenience we assume that $\alpha, \beta > 0$.
We point out that the theory  depends heavily on the
assumption that $\Lambda$ is a separable lattice $\latt$.

We first recall the definition of the family of sequence spaces corresponding
to amalgam spaces via Gabor frames.  For a weight $v$ and $1 \leq p,q \leq
+\infty$ we define the sequence space
$\seqpqvl$ in the following
way. We let $\Fou\LpQb$ stand for the image of $\LpQb$ under the discrete
Fourier transform. More precisely, a sequence
$c \equiv \set{c_j}{j \in \beta \Zdst} \subseteq \bC$
belongs to $\Fou \LpQb$ if there exists a (unique) function
$f \in \LpQb$ such that,
\begin{align*}
c_j = \hat{f}(j)=\beta^d \int_\Qb f(x) e^{-2 \pi i j x} dx,
\qquad j \in \beta \Zdst.
\end{align*}
The space $\Fou\LpQb$ is given the norm $\norm{c}_{\Fou\LpQb} :=
\norm{f}_\LpQb$.

We now let $\seqpqvl$ be the set of all sequences
$c\equiv \set{c_\lambda}{\lambda \in \Lambda} \subseteq \bC$ such that,
for each $k \in \alpha\Zdst$, the sequence $(c_{k, j})_{j\in\beta\Zdst}$ belongs
to
$\Fou\LpQb$ and
\begin{align*}
\norm{c}_\seqpqvl :=
\left(
\sum_{k\in\alpha\Zdst}
\bignorm{(c_{k, j})_{j\in\beta\Zdst}}_{\Fou\LpQb}^q v(k)^q
\right)^{1/q}
<+\infty,
\end{align*}
with the usual modifications when $q=\infty$. When $1<p<+\infty$ this is simply,
\begin{align*}
\norm{c}_\seqpqvl :=
\left(
\sum_{k\in\alpha\Zdst}
\Bignorm{\sum_{j \in \beta\Zdst} c_{k,j} e^{2\pi ij \cdot}}_{\LpQb}^q v(k)^q
\right)^{1/q}
<+\infty,
\end{align*}
and the usual modifications hold for $q=\infty$.

The following Theorem from \cite{grheok02}
introduces the analysis and synthesis operators, clarifies their precise meaning
and gives their mapping properties.
\begin{theo}\label{th_gabop} \cite[Theorem 3.2]{grheok02}.
Let $w$ be a submultiplicative weight, $v$ a $w$-moderate weight,
$g \in \winw$, and $1 \leq p,q \leq +\infty$. Then the following properties hold.
\begin{itemize}
\item[(a)] The analysis (coefficient) operator,
\begin{align*}
\coefgl: \winpqv \to \seqpqvl,\qquad
\coefgl(f) := \left( \ip{f}{\pi(\lambda)g} \right)_{\lambda \in \Lambda} 
\end{align*}
is bounded with a bound that only depends on $\alpha, \beta,
\norm{g}_\winw$, and the constant $\consv$ in~\eqref{eq_consv}.
\item[(b)] Let $c \in \seqpqvl$ and  $m_k \in \LpQb$ be the unique
functions such that $\widehat{m_k}(j) = c_{k,j}$. Then the series
\begin{align*}
\recgl(c) := \sum_{k \in \alpha \Zdst} m_k T_k g,
\end{align*}
converges unconditionally in the $\koethetoppqv$-topology and, moreover, 
unconditionally in the norm topology of $\winpqv$ if $p,q < \infty$.
\item[(c)] The synthesis (reconstruction) operator
$\recgl: \seqpqvl \to \winpqv$ is bounded with a bound that  depends only on
$\alpha, \beta, \norm{g}_\winw$, and the constant $\consv$ in~\eqref{eq_consv}.
\end{itemize}
\end{theo}
The definition of the operator $\recgl$ is rather abstract. As shown in
\cite{fewe07}, the convergence can be made explicit by means of a summability
method.

For $g \in \winw$, a sequence $c \in \seqpqvl$, and $N,M \geq 0$ let us
consider the partial sums
\begin{align*}
\gabpartsums_{N,M}(c)(x) :=
\sum_{\abs{k}_\infty \leq \alpha N}
\sum_{\abs{j}_\infty \leq \beta M}
c_{k,j} e^{2 \pi i j x} g(x-k).
\end{align*}
In the conditions ``$\abs{k}_\infty \leq N, \abs{j}_\infty \leq M$'' above
we consider elements $(k,j) \in \Lambda=\latt$; it is important that we use the
max norm. We also consider the regularized partial sums,
\begin{align*}
\gabregpartsums_{N,M}(c)(x) :=
\sum_{\abs{k}_\infty \leq \alpha N}
\sum_{\abs{j}_\infty \leq \beta M}
\regwegjM c_{k,j} e^{2 \pi i j x} g(x-k),
\end{align*}
where the \emph{regularizing weights} are given by,
\begin{align}
\label{eq_reg_weights}
\regwegjM := \prod_{h=1}^d\left( 1 - \frac{\abs{j_h}}{\beta(M+1)} \right).
\end{align}
We then have the following convergence result \cite{fewe07, grheok02}.

\begin{theo}
\label{th_gab_summ}
Let $w$ be a submultiplicative weight, $v$ a $w$-moderate weight,
$g \in \winw$, and $1 \leq p,q \leq +\infty$. Then the following properties hold.
\begin{itemize}

\item[(a)] If $1<p<\infty$ and $q<\infty$, then 
\begin{align*}
\gabpartsums_{N,M}(c) \rightarrow \recgl(c),
\qquad
\mbox{as }N,M \rightarrow \infty, 
\end{align*}
in the norm of $\winpqv$.
\item[(b)] For each $c \in \seqpqvl$, 
\begin{align*}
\gabregpartsums_{N,M}(c) \rightarrow \recgl(c),
\qquad
\mbox{as }N,M \rightarrow \infty, 
\end{align*}
in the $\koethetoppqv$-topology and also in the norm of $\winpqv$ if $p,q
<+\infty$.
\end{itemize}
\end{theo}
\begin{remark}
A more refined convergence statement, with more general summability methods,
can be found in \cite{fewe07}. We will only need the norm and weak
convergence of Gabor expansions but we point out that the problem of pointwise
summability has also been extensively studied \cite{fewe07, grle01, grhe01, grheok02,
we09-1}.
\end{remark}
\begin{proof}
Part (a) is proved in \cite[Proposition 4.6]{grheok02}. The case $p<+\infty$ of (b)
is proved in \cite[Theorem 4]{fewe07}, where only unweighted amalgam spaces are considered.
The same proof extends with simple modifications to the weighted case.
The case $p=+\infty$ is also
treated in \cite{fewe07} but in a different direction, establishing
norm convergence under additional assumptions on the sequence $c$.
Hence, we sketch a proof of (b) when $p=+\infty$.

Let $c \in \seqinfqv(\Lambda)$. According to Theorem \ref{th_gabop},
$\recgl(c)$ is given by
\begin{align}
\label{eq_recc}
\recgl(c) = \sum_{k \in \alpha \Zdst} m_k T_k g,
\end{align}
where $m_k \in \Linf([0,1/\beta)^d)$ and $\norm{c}_\seqinfqv = \bignorm{(
\norm{m_k}_\infty )_k}_{\ell^q_v}$.
Let us also write
\begin{align*}
\gabregpartsums_{N,M}(c)
:= \sum_{\abs{k}_\infty \leq \alpha N} 
\regM(m_k) T_k g,
\end{align*}
where $\regM(m_k)(x) = 
\sum\limits_{\abs{j}_\infty \leq \beta M}
\regwegjM c_{k,j} e^{2\pi i j x}$
are the Fej\'er means of $m_k$.

Let $f \in \wininfqvdual$, we must show that
$\ip{\recgl(c)-\gabregpartsums_{N,M}(c)}{f} \rightarrow 0$.
Let us write $f = \sum_{j \in \alpha \Zdst} f_j$
and $g= \sum_{j \in \alpha \Zdst} g_j$,
where $f_j=f \chi_{[0,\alpha)^d+j}$
and $g_j=g \chi_{[0,\alpha)^d+j}$.
Hence, for any subset $L \subseteq \alpha\Zdst$,
\begin{align}
\label{eq_bracket}
\ip{\sum_{k\in L} m_k T_k g}{f}
=
\sum_{k \in L} \sum_{j,j' \in \alpha \Zdst} \ip{m_k T_k g_j}{f_{j'}}
=
\sum_{k \in L} \sum_{j \in \alpha \Zdst}
\ip{m_k T_k g_j}{f_{j+k}}.
\end{align}
In addition,
\begin{align}
\label{eq_bound_bracket}
&\sum_{j \in \alpha \Zdst}
\sum_{k \in L} \abs{\ip{m_k T_k g_j}{f_{j+k}}}
\\
\nonumber
&\quad\lesssim
\sum_{k \in L} \sum_{j \in \alpha\Zdst}
\norm{g_j}_\infty w(j)
\norm{m_k}_\infty v(k)
\norm{f_{j+k}}_1 (1/v)(j+k)
\\
\nonumber
&\quad\lesssim \norm{g}_\winw \norm{c}_\seqinfqv \norm{f}_\wininfqvdual.
\end{align}
Hence, given $\varepsilon>0$, there exists $N_0 >0$ 
and a finite set $J \subseteq \alpha\Zdst$ such that
\begin{align}
\label{eq_bound_bracket3}
\abs{\ip{\recgl(c)}{f} - 
\sum_{\abs{k} \leq \alpha N_0} 
\sum_{j \in J}
\ip{m_k T_k g}{f}}<\varepsilon.
\end{align}
Similarly, using the fact
$\norm{\regM(m_k)}_\infty \lesssim \norm{m_k}_\infty$ independently of $M$,
it follows that $N_0$ can be chosen so that, in addition,
\begin{align}
\label{eq_bound_bracket4}
\abs{\gabregpartsums_{N,M}(c)-
\sum_{\abs{k} \leq \alpha N_0} 
\sum_{j \in J} \ip{\regM(m_k) T_k g}{f}}
<\varepsilon,
\end{align}
for all $M \geq 0$ and $N \geq N_0$.
Finally, for each $j,k$,
\begin{align*}
\ip{(m_k - \regM(m_k)) T_k g_j}{f_{j+k}}
=\ip{m_k - \regM(m_k)}{\overline{T_k g_j}f_{j+k}}
\rightarrow 0,
\end{align*}
as $M \rightarrow +\infty$ because $\overline{T_k g_j}f_{j+k} \in L^1$
and the Fej\'er means of an $L^\infty$ function converge in the
weak*-topology. This, together with the estimates in~\eqref{eq_bound_bracket3}
and~\eqref{eq_bound_bracket4}, yields
the desired conclusion.
\end{proof}

We now present a representation of Gabor frame operators that will be essential
for the results to come. For proofs see \cite{wa92} or
\cite[Theorem, 4.2 and Lemma 5.2]{grheok02} for the weighted version.
\begin{theo}
\label{th_walnut_rep}
Let $w$ be a submultiplicative weight, $v$ a $w$-moderate weight,
$g,h \in \winw$ and $1 \leq p,q \leq +\infty$. Then 
the operator
$\rechl \coefgl: \winpqv \to \winpqv$ can be written as
\begin{align}
\label{eq_walnut1}
\rechl \coefgl f = \beta^{-d} \sum_{j \in \Zdst} G_j T_{\frac{j}{\beta}} f,
\end{align}
where,
\begin{align}
\label{eq_walnut2}
G_j(x) := \sum_{k \in \Zdst}
\overline{g(x- j/\beta-\alpha k)} h(x-\alpha k),
\qquad x\in\Rdst.
\end{align}
In addition, the functions $G_j:\Rdst \to \bC$ satisfy
\begin{align}
\label{eq_walnut3}
\sum_{j \in \Zdst} \norm{G_j}_\infty w(j/\beta)
\lesssim \norm{g}_\winw \norm{h}_\winw <+\infty.
\end{align}
As a consequence, the series in~\eqref{eq_walnut1} converges
absolutely in the norm of $\winpqv$.
\end{theo}

\section{The algebra of $L^\infty$-weighted shifts}\label{sec3}
\label{sec_operators}
\subsection{$L^\infty$-weighted shifts}
Guided by \eqref{eq_walnut1}, we will now introduce a Banach *-algebra of operators on function spaces that will be the key
technical object of the article. For an admissible weight $w$ we let
$\opAw$ be the set of all families $\mathcal M = (m_x)_{x\in \Rdst} \in \lwLinfRd$ with the standard Banach space norm
\begin{align}
\label{eq_opAw}
\|\mathcal M\|_{\opAw}=\sum_{x \in \Rdst} \norm{m_x}_{\LinfRd} w(x) < +\infty.
\end{align}
The algebra structure 
and the involution on $\opAw$, however, will be non-standard.
They will come from the
 identification of $\opAw$ with the class of operators on function spaces of the form
\begin{align}
\label{eq_deff_sum_mx_tx}
f \mapsto \sum_{x \in \Rdst} m_x f(\cdot-x).
\end{align}
Observe that due to \eqref{eq_opAw} the family 
$\mathcal M = (m_x)_{x\in \Rdst}$ has countable support and also that
the operator in~\eqref{eq_deff_sum_mx_tx} is well-defined
and bounded on all $\LpRd$, $p\in[1,\infty]$ (recall that the admissibility of $w$ implies
that $w \geq 1$).

With a slight abuse of notation, given a function $m \in L^\infty(\Rdst)$ we also denote by $m$ the
multiplication operator $f \mapsto mf$. It is then convenient to write $\mathcal M\in\opAw$ as
$${\mathcal M }= \sum_{x \in \Rdst} m_x T_x, \quad (m_x)_{x\in \Rdst} \in \lwLinfRd,$$
and endow $\opAw$ with the product and involution inherited from $B(\LtRd)$.
More precisely, the product on $\opAw$ is given by
\begin{align*}
\left( \sum_x m_x T_x \right)
\left( \sum_x n_x T_x \right)
=
\sum_x \left( \sum_y m_y n_{x-y}(\cdot -y) \right) T_x,
\end{align*}
and the involution -- by
\begin{align*}
\left( \sum_x m_x T_x \right)^*
= \sum_x \overline{m_x(\cdot +x)} T_{-x}
= \sum_x \overline{m_{-x}(\cdot -x)} T_x.
\end{align*}

It is straightforward to verify that with this structure $\opAw$ is, indeed, a
Banach *-algebra which embeds continuously into $B(\LtRd)$. 
We shall establish a number of other continuity properties of the operators
defined by families in
$\opAw$ in Proposition \ref{prop_mapping} below. These will be useful in dealing with Gabor expansions on
amalgam spaces.

 Before that,
we mention that the identification of families in $\opAw$ and operators on $B(\LpRd)$ given by the operator in~\eqref{eq_deff_sum_mx_tx} is one-to-one; this follows from the characterization of $\opAw$ in the following subsection and can easily be proved directly. Because of this we shall no longer distinguish between the families in $\opAw$ and operators generated by them. We will write  $\opAw\subset B(\LpRd)$ if we need to highlight that we treat members of $\opAw$ as operators on $\LpRd$. We also point out that for $m \in \LinfRd$ and $x,w \in \Rdst$
\begin{equation}\label{obs_conj_mw}
M_{\omega} m T_x M_{-\omega} = e^{2 \pi i \omega \cdot x} m T_x. 
\end{equation}

\begin{prop}
\label{prop_mapping}
Let $1 \leq p,q \leq +\infty$ and let $v$ be a $w$-moderate weight. Then the
following statements hold.
\begin{itemize}
\item[(a)] $\opAw \into B(\winpqv)$. More precisely,
every  ${\mathcal M}=\sum_x m_x T_x \in \opAw$ defines a bounded
operator on $\winpqv$ given by the formula
\begin{align*}
{\mathcal M}(f) :=  \sum_x m_x f(\cdot-x).
\end{align*}
The series defining ${\mathcal M}: \winpqv \to \winpqv$ converges absolutely in
the norm of $\winpqv$ and
$\norm{{\mathcal M}}_{B(\winpqv)} \leq \consv \norm{{\mathcal M}}_\opAw$,
where $\consv$ is the constant in \eqref{eq_consv}.
\item[(b)] For every ${\mathcal M} \in \opAw$, $f \in \winpqv$
and $g \in \winpqvdual$,
\begin{align*}
\ip{{\mathcal M}(f)}{g}=\ip{f}{{\mathcal M}^*(g)}. 
\end{align*}
\item[(c)] For every ${\mathcal M} \in \opAw$, the operator
${\mathcal M}: \winpqv \to \winpqv$ is continuous in the
$\koethetoppqv$-topology.
\end{itemize}
\end{prop}
\begin{proof}
Part (a) follows immediately from~\eqref{eq_solid_winpqv} and~\eqref{eq_trans_winpqv}. Part (b) follows from the fact the involution in
$\opAw$ coincides with taking adjoint. The interchange of summation and
integration is justified by the absolute convergence in part (a). Part (c)
follows immediately from (b).
\end{proof}

\subsection{Spectral invariance}
In this section we shall exhibit the main technical result of the article. We  remark that similar and more general results appear in \cite{Bas92, Bas97vuz, K99}. We, however, feel obliged to present a proof here because the rest of our paper is based on this result.
The key ingredient in the proof is the identification of the algebra $\opAw$ with  a class of
almost periodic elements associated with a certain group representation.
We give a brief account of the theory as required for our purposes.
For a more general presentation see \cite{ba10-1} and references therein. 

For $y \in \Rdst$ and $\Mou \in B(\LpRd)$, $p\in[1,\infty]$, 
let $\rep(y)\Mou := M_y \Mou M_{-y}$. Explicitly,
\[ \rho(y)\Mou f(x) = e^{2\pi i y\cdot x}(\Mou g)(x)~~,~~
g(x)=e^{-2\pi i y\cdot x}f(x). \]
The map $\rep: \Rdst \to B(B(\LpRd))$ defines an isometric representation
of $\Rdst$ on the algebra $B(\LpRd)$. This means that $\rep$ is a
representation
of $\Rdst$ on the Banach space $B(\LpRd)$ and, in addition, for each $y \in
\Rdst$, $\rep(y)$ is an algebra automorphism and an isometry.

A continuous map $Y:\Rdst\rightarrow B(\LpRd)$ is   {\em almost-periodic
in the sense of Bohr} if for every $\eps>0$ there is a compact $K=K_\eps
\subset\Rdst$ such that for all $x\in\Rdst$ 
\[ (x+K)\cap \{ y\in\Rdst~|~\norm{Y(g+y)-Y(g)}<\eps,~\forall g\in\Rdst\}
\neq \emptyset \]
Then $Y$ extends uniquely to a continuous map of the Bohr compactification
$\Rdc$ of $\Rdst$, also denoted by $Y$. Thus, now
$Y:\Rdc\rightarrow B(\LpRd)$, where
$\Rdc$ represents the topological dual group (i.e. the group of characters)
of $\Rdst$ when $\Rdst$ is endowed with the discrete topology. The normalized
Haar measure on $\Rdc$ is denoted by $\bar{\mu}(dy)$.

For each $\Mou \in B(\LpRd)$, we consider the map,
\begin{align}
\widehat \Mou: \Rdst \to B(\LpRd),
\qquad \widehat \Mou(y) := \rep(y)\Mou = M_y \Mou M_{-y}.
\label{eq:e18}
\end{align}

An operator $\Mou \in B(\LpRd)$ is said to be $\rep$-\emph{almost periodic}
if the map $\widehat \Mou$ is continuous and almost-periodic in the sense 
of Bohr. For every $\rep$-almost periodic
operator $\Mou$, the function $\widehat \Mou$ admits 
a $B(\LpRd)$-valued Fourier series,
\begin{align}
\label{eq_ap_fourier}
\widehat \Mou(y) \sim \sum_{x \in \Rdst} e^{2 \pi i y\cdot x} C_x (\Mou),
\qquad (y \in \Rdst).
\end{align}
The coefficients $C_x(\Mou) \in B(\LpRd)$ in~\eqref{eq_ap_fourier} 
are uniquely determined by $\Mou$ via
\begin{align}
\label{eq:Cx}
C_x(\Mou) = \int_{\Rdc}\widehat{\Mou}(y) e^{-2\pi i y\cdot x}\bar{\mu}(dy)
=\lim_{T\to\infty} \frac1{(2T)^d}\int_{[-T,T]^d}\widehat{\Mou}(y) e^{-2\pi i y\cdot x}dy
\end{align}
and, therefore, satisfy
\begin{align}
\label{eq_eigen}
\rep(y) C_x (\Mou) = e^{2 \pi i y\cdot x} C_x (\Mou).
\end{align}
Hence, they are eigenvectors of $\rep$ (see \cite{ba10-1} for details).

Within the class of $\rep$-almost periodic operators we consider
$\apw$, the subclass of those operators for which the Fourier
series in~\eqref{eq_ap_fourier} is $w$-summable,
where $w$ is an admissible weight. More precisely, a $\rep$-almost periodic
operator $\Mou$ belongs to $\apw$ if its Fourier coefficients with respect to
$\rep$ satisfy
\begin{align}
\label{eq_apw}
\norm{\Mou}_{\apw} :=
\sum_{x \in \Rdst} \norm{C_x(\Mou)}_{B(\LpRd)} w(x) < + \infty.
\end{align}
By the submultiplicativity of $w$ we know that $w \geq 1$, 
so for operators in $\apw$ the series in~\eqref{eq_ap_fourier}
converges absolutely in the norm of $B(\LpRd)$ to $\widehat \Mou(y)$:
\begin{align}
\widehat{\Mou}(y) = \sum_{x\in\Rdst} e^{2 \pi i y\cdot x} C_x (\Mou), 
\qquad y \in \Rdst,
\end{align}
where each $C_x \in B(\LpRd)$ satisfies~ \eqref{eq:Cx} and, hence, \eqref{eq_eigen}.
In particular, for $y=0$, it follows that each $\Mou \in \apw$ can
be written as
\begin{align}
\label{eq_T_Cx}
\Mou = \sum_{x \in \Rdst} C_x (\Mou).
\end{align}
Conversely, if $\Mou$ is given by~\eqref{eq_T_Cx}, with
the coefficients $C_x$ satisfying~\eqref{eq_apw} and~\eqref{eq_eigen},
it follows from the theory of almost-periodic series that
$\Mou \in \apw$ and $C_x$ satisfy \eqref{eq:Cx}.

Theorem 3.2 from \cite{ba10-1} establishes the spectral invariance
of $\apw \into B(\LpRd)$, $p\in[1,\infty]$ 
(the result there applies to a more general
context). Our goal here is to establish  connection between $\opAw$ 
and $\apw$ and prove a spectral invariance result for $\opAw$.

To achieve this goal we first characterize the eigenvectors $C_x$ of the 
representation $\rep$.
\begin{lemma}
\label{lemma_eigen}
For any $1\leq p\leq \infty$ and any $m\in\LinfRd$
and $x\in\Rdst$, $C_x=mT_x$ is an eigenvector of
$\rep:\Rdst\rightarrow B(\LpRd)$.
For $1\leq p <\infty$ these are the only eigenvectors.
\end{lemma}
\begin{proof}
If $C_x = m T_x$, then, according to~\eqref{obs_conj_mw},
it satisfies~\eqref{eq_apw}. 

The converse works only for $1\leq p<\infty$.
Suppose that $C_x \in B(\LpRd)$ satisfies~\eqref{eq_apw}. 
Using ~\eqref{obs_conj_mw} once again we have,
\[ \rep(y)(C_x T_{-x})=e^{2\pi i y\cdot x} C_x e^{-2\pi i y\cdot x} T_{-x} =C_x T_{-x}. \]
It follows that $C_x T_{-x}$ commutes
with every modulation $M_y$. Hence, $C_x T_{-x}$ must be a multiplication
operator $m$, so $C_x= m T_x$.
\end{proof}

For $p=\infty$ there are eigenvectors of $\rep$ which are not of the form
$m T_x$. An example of such an eigenvector is given in \cite[Section 5.1.11]{K99}. Hence, one would  need  additional conditions
to conclude that $C_x= m T_x$ for some $m\in\LinfRd$.

From the discussion above, $\apw$ consists of all the operators
$\Mou = \sum\limits_{x \in \Rdst} C_x$, with $C_x$ 
satisfying~\eqref{eq_apw} and~\eqref{eq_eigen}. 
In addition, by the previous lemma, for $1\leq p<\infty$ an operator
$C_x$ satisfies~\eqref{eq_eigen} if and only if it is of the form
$C_x = m T_x$, for some function $m \in\LinfRd$. In this case, $\norm{C_x}_{B(\LtRd)} = \norm{m}_\infty$ and, thus,
~\eqref{eq_apw} reduces to~\eqref{eq_opAw}. Hence we obtained

\begin{prop}\label{plinf}
For $p\in [1,\infty)$ the class $\opAw\subset B(\LpRd)$ coincides with $\apw$,
  the class of $\rep$-almost periodic
elements, having $w$-summable Fourier coefficients.
\end{prop}

For $p=\infty$, the two classes are different. Nevertheless, the results we have
obtained so far are sufficient to prove our main technical result.

\begin{theo}
\label{th_spectral}
Let $w$ be an admissible weight. Then, the embedding 
$\opAw \into B(\LpRd)$, $p\in[1,\infty]$
is spectral. In other words, if ${\mathcal M} \in \opAw$ 
defines an invertible operator $\sum_x m_xT_x\in
B(\LpRd)$ for some $p\in[1,\infty]$, then ${\mathcal M}^{-1} \in \opAw$. 
\end{theo}
\begin{proof}
For $1\leq p<\infty$ the result follows from Proposition 
\ref{plinf} and \cite[Theorem 3.2]{ba10-1}. This last result states
that $\apw$ is spectral.

For $p=\infty$ we follow a different path. Given an operator
\[ \Mou = \sum_{x \in \Rdst} m_x T_x \in\opAw\subset  B(L^{\infty}(\Rdst)) \]
with $\sum_{x\in\Rdst}w(x)\norm{m_x}_{\LinfRd} <\infty$,
we consider the  operator
\[ \Nou = \sum_{x\in\Rdst} T_x(m_{-x})T_x  = \sum_{x\in\Rdst} m_{-x}(\cdot-x)T_x \in \opAw\subset B(L^1(\Rdst)), \]
which is well defined since 
$\norm{T_{x}(m_{-x})}_{\LinfRd}=\norm{m_{-x}}_{\LinfRd}$.
By direct computation, the transpose (Banach adjoint) of 
$\Nou:L^1(\Rdst)\rightarrow L^1(\Rdst)$ is
precisely $\Mou:\LinfRd\rightarrow\LinfRd$. Thus,
$\Mou=\Nou'$ and by \cite[Theorem 3, Chapter 20]{Lax02} it follows
that $\Nou$ is invertible when $\Mou$ is invertible. Now,
by spectrality of $\opAw$ in $B(L^1(\Rdst))$ (as obtained earlier) and 
\cite[Theorem 8(ii), Chapter 15]{Lax02}, we obtain that
$\Mou^{-1}=(\Nou^{-1})'\in\opAw$, that is $\Mou^{-1}=\sum\limits_{x\in\Rdst}n_xT_x$
for some bounded functions $n_x$ such that 
$\sum\limits_{x\in\Rdst}w(x)\norm{n_x}_{\LinfRd}<\infty$.
\end{proof}

\begin{remark}
In concrete terms, Theorem \ref{th_spectral} says that if 
$\mathcal{M}:L^p(\Rdst) \to L^p(\Rdst)$ is an invertible operator of the form
$\mathcal{M}=\sum_{x\in \Rdst} m_x T_x$ with
$\{m_x: x \in \Rdst\} \subseteq L^\infty(\Rdst)$ and 
$\sum_x \norm{m_x}_\infty w(x) < +\infty$, for an admissible weight $w$, then 
$\mathcal{M}^{-1}:L^p(\Rdst) \to L^p(\Rdst)$ can also be written as
$\mathcal{M}^{-1}=\sum_{x\in \Rdst} n_x T_x$,
for some measurable functions $n_x,x \in \Rdst$ satisfying
$\sum_x \norm{n_x}_\infty w(x) < +\infty$.
\end{remark}

\begin{remark}
In \cite{krok08} two of us used a special case of
Theorem~\ref{th_spectral} for $\rho$-periodic (rather than $\rho$-almost periodic) operators in $B(L^{2}(\Rdst))$.
 In   \cite[Example 2.1]{krok08}, however, we neglected to mention this restriction and erroneously implied   that all of the operators in $B(L^{2}(\Rdst))$ were
 $\rho$-periodic. 
\end{remark}

\subsection{Corollaries of spectral invariance}
Let us denote by $\sigma_p(\Mou)$ and $\sigma_\opAw(\Mou)$ the spectra of the
operator $\Mou\in\opAw$ in the algebras $B(\LpRd)$, $p\in[1,\infty]$, and
$\opAw$, respectively.

\begin{coro}\label{coro1}
Consider $\Mou = \sum_x m_xT_x\in\opAw$. 
Then $\sigma_p(\Mou) = \sigma_\opAw(\Mou)$ for all $p\in[1,\infty]$. 
\end{coro}

We conclude the section with the following very important
result.

\begin{theo}\label{th_opAw_inv_sub}
Assume that $\Mou \in \opAw$ satisfies $\Mou^* = \Mou = \sum_x m_xT_x$ and
$ A_{r}\norm{f}_r \leq  \norm{\Mou f}_r$  for some $A_r > 0$ and all $f\in L^r(\Rdst)$ for some $r\in [1,\infty]$.
Then $\Mou^{-1}\in\opAw$.

Moreover, suppose 
that $\Esp \subseteq \winpqv$,  $1 \leq p,q \leq +\infty$, is a closed subspace (in the norm of $\winpqv$)
such that $\Mou \Esp \subseteq \Esp$. Then $\Mou^{-1} \Esp \subseteq \Esp$
and, as a consequence, $\Mou \Esp = \Esp$.
\end{theo}

\begin{proof}
From Corollary \ref{coro1} we deduce that 
$\sigma_\opAw(\Mou) = \sigma_r(\Mou)= \sigma_2(\Mou)\subset \RR$ since 
$\Mou\in\BLtRd$ is self-adjoint. 
Recall that in Banach algebras every boundary point of the spectrum belongs to the approximative spectrum. The boundedness below condition, however, implies that $0$ does not belong to the approximative spectrum of $\Mou\in B(L^r(\Rdst))$. Hence, $0\notin\sigma_r(\Mou)$ and, by Theorem \ref{th_spectral}, $\Mou^{-1}\in\opAw$.

To prove the second part, let $\opAwE$ be the subalgebra of $\opAw$ formed by all those operators $S$
such that $S\Esp \subseteq \Esp$. Since $\Esp$ is closed in $\winpqv$
and $\opAw \into B(\winpqv)$ by Proposition~\ref{prop_mapping}, it follows
that $\opAwE$ is a closed subalgebra of $\opAw$ (we do not claim that
it is closed under the involution). From the first part of the proof it follows that the set
$\bC\setminus\sigma_{\opAw}(\Mou)$ 
is connected. Consequently,
(see for example \cite[Theorem VII 5.4]{co90}),
$\sigma_{\opAwE}(\Mou) = \sigma_{\opAw}(\Mou)$.
Finally,
$0 \notin  \sigma_{\opAw}(\Mou)= \sigma_{\opAwE}(\Mou)$ which
proves that $\Mou^{-1} \in \opAwE$, as desired.
\end{proof}

\section{Dual Gabor frames on amalgam spaces}\label{sec4}
\subsection{Multi-window Gabor frames}
Let $\Lambda=\Lambda^1 \times \ldots \times \Lambda^n$ be the Cartesian
product of separable lattices $\Lambda^i=\latti$ and let
$g^1, \ldots, g^n \in \winw$. We consider the (multi-window) Gabor system
\begin{align*}
\gab = \set{\gili:=\pi(\lambda^i)g^i}{\lambda^i \in \Lambda^i, 1 \leq i \leq n}.
\end{align*}
We consider the system $\gab$ as an indexed set, hence $\gab$ might contain
repeated elements. The frame operator of the system $\gab$ is given by,
\begin{align*}
\opgab = \opgabo + \ldots \opgabn,
\end{align*}
where $\opgabi = \recgili \coefgili$ (see Section \ref{sec_gab_lat}).
For $1 \leq p,q \leq +\infty$ and a $w$-moderate weight $v$, we define the
space $\seqpqvl := \seqpqvlo \times \ldots \times \seqpqvln$
endowed with the norm,
\begin{align*}
\norm{c=(c^1,\ldots,c^n)}_\seqpqvl := \sum_{i=1}^n \norm{c^i}_\seqpqvli.
\end{align*}
The analysis map is
$\winpqv \ni f \mapsto \coefgab(f):=
(\coefgili(f))_{1 \leq i \leq n} \in \seqpqvl$,
while the synthesis map is
$\seqpqv \ni c \mapsto \recgab(c) :=
\sum_{i=1}^n \recgili(c^i) \in \winpqv$.
With these definitions, the boundedness results in Theorem \ref{th_gabop}
extend immediately to the multi-window case. The frame expansions are however
more complicated since the dual system of a frame of the form of $\gab$ may
not be a multi-window Gabor frame. We now investigate this matter.
\subsection{Invertibility of the frame operator and expansions}
\begin{theo}
\label{th_inv_gab_op}
Let $w$ be an admissible weight, $g^1, \ldots, g^n \in \winw$
and $\Lambda=\Lambda^1 \times \ldots \times \Lambda^n$, with
$\Lambda^i=\latti$ separable lattices.
Suppose that the Gabor system
\begin{align*}
\gab = \set{\gili:=\pi(\lambda^i)g^i}{\lambda^i \in \Lambda^i, 1 \leq i \leq n},
\end{align*}
is 
such that its frame operator $\opgab$ is bounded below in some $L^r(\Rdst)$ for 
some $r\in[1,\infty]$, i.e.
\[A_r\|f\|_r \le  \|\opgab f\|_r,\ A_r > 0, \quad \mbox{ for all } f\in L^r(\Rdst).\]

Then the frame operator $\opgab$ is invertible on
$\winpqv$ for all $1 \leq p,q \leq +\infty$ and every $w$-moderate weight
$v$. Moreover, the inverse operator $\opgabinv: \winpqv \to \winpqv$ is
continuous both in  $\koethetoppqv$ and the norm topologies.
\end{theo}
\begin{proof}
For each $1 \leq i \leq n$, the frame operator
$\opgabi = \recgili \coefgili$ belongs to the algebra $\opAw$ as a
consequence of the Walnut representation in Theorem \ref{th_walnut_rep}.
Hence, $\opgab = \opgabo + \ldots \opgabn \in \opAw$. 
Since $\opgab$ is bounded below in  $L^r(\Rdst)$,
Theorem \ref{th_opAw_inv_sub} implies that $\opgabinv \in \opAw$. The conclusion
now follows from Proposition \ref{prop_mapping}.
\end{proof}
We now derive the corresponding Gabor expansions.
\begin{theo}
Under the conditions of Theorem \ref{th_inv_gab_op}, define the dual atoms
by $\tilgili := \opgabinv ( \gili )$. Let $1 \leq p,q \leq +\infty$ and $v$
be a $w$-moderate weight. Then the following expansions hold.
\begin{itemize}
\item[(a)] For every $f \in \winpqv$,
\begin{align*}
f &= \lim_{N,M \rightarrow \infty}
\sum_{i=1}^n \sum_{\abs{k}_\infty \leq N} \sum_{\abs{j}_\infty \leq M}
\regwegbijM \ip{f}{\tilgilikj} \gilikj
\\
&= \lim_{N,M \rightarrow \infty}
\sum_{i=1}^n \sum_{\abs{k}_\infty \leq N} \sum_{\abs{j}_\infty \leq M}
\regwegbijM \ip{f}{\gilikj} \tilgilikj,
\end{align*}
where the regularizing weights $\regwegbijM$ are given in~\eqref{eq_reg_weights} and the series converge in the
$\koethetoppqv$-topology. For $p,q < +\infty$ the series also converge
in the norm of $\winpqv$.
\item[(b)] If $1<p<+\infty$ and $q<+\infty$, for every $f \in \winpqv$,
\begin{align*}
f &= \lim_{N,M \rightarrow \infty}
\sum_{i=1}^n \sum_{\abs{k}_\infty \leq N} \sum_{\abs{j}_\infty \leq M}
\ip{f}{\tilgilikj} \gilikj
\\
&= \lim_{N,M \rightarrow \infty}
\sum_{i=1}^n \sum_{\abs{k}_\infty \leq N} \sum_{\abs{j}_\infty \leq M}
\ip{f}{\gilikj} \tilgilikj,
\end{align*}
where the series converge in the in the norm of $\winpqv$.
\end{itemize}
\end{theo}
\begin{remark}
A more refined convergence statement, including more sophisticated
summability methods can be obtained using the results in \cite{fewe07}.
\end{remark}
\begin{proof}
Theorem \ref{th_gab_summ} implies that for all $f \in \winpqv$,
\begin{align}
\label{eq_gab_frame_op}
\opgab (f) = \lim_{N,M \rightarrow \infty}
\sum_{i=1}^n \sum_{\abs{k}_\infty \leq N} \sum_{\abs{j}_\infty \leq M}
\regwegbijM \ip{f}{\gilikj} \gilikj,
\end{align}
with the kind of convergence required in (a).
Since $\opgabinv \in \opAw$, Proposition \ref{prop_mapping} implies that
$\opgabinv:\winpqv \to \winpqv$ is continuous both in
the norm and $\koethetoppqv$-topology. Consequently, we
can apply $\opgabinv$ to both sides 
of~\eqref{eq_gab_frame_op} to obtain the first expansion in (a). The
second one follows by applying~\eqref{eq_gab_frame_op} to the
function $\opgabinv(f)$ and using Proposition \ref{prop_mapping} to get,
\begin{align*}
\ip{\opgabinv(f)}{\gili}=\ip{f}{\opgabinv(\gili)}=\ip{f}{\tilgili}.
\end{align*}
The statement in (b) follows similarly, this time using the corresponding
statement in Theorem \ref{th_gab_summ}.
\end{proof}
\subsection{Continuity of dual generators}
We now apply Theorem \ref{th_opAw_inv_sub} to Gabor expansions.
\begin{theo}
\label{th_inv_sub}
In the conditions of Theorem \ref{th_inv_gab_op}, let $1 \leq p,q \leq
+\infty$ and let $v$ be a $w$-moderate weight. Let
$\Esp \subseteq \winpqv$ be a closed subspace (in the norm of $\winpqv$)
such that $\opgab \Esp \subseteq \Esp$.
Suppose that the atoms $g^1, \ldots, g^n \in \Esp$. Then the dual atoms,
$\tilgili = \opgabinv(\gili) \in \Esp$.
\end{theo}
\begin{proof}
As seen in the proof of Theorem \ref{th_inv_gab_op}, $\opgab \in \opAw$.
Hence, the conclusion follows from Theorem \ref{th_opAw_inv_sub}.
\end{proof}
As an application of Theorem \ref{th_inv_sub}
we obtain the following corollary, which was one of our main motivations.
The case $n=1$ was an open problem in \cite{krok08}.
\begin{coro}\label{contidual}
In the conditions of Theorem \ref{th_inv_gab_op}, if all the atoms
$g^1, \ldots, g^n$ are continuous functions, so are all the dual atoms
$\tilgili = \opgabinv(\gili)$.
\end{coro}
\begin{proof}
We apply Theorem \ref{th_inv_sub} to the subspace $\winwz$ formed by the
functions of $\winw$ that are continuous. To this end we need to observe that
$\opgab \winwz \subseteq \winwz$. Since $\opgab = \opgabo + \ldots \opgabn$,
it suffices to show that each $\opgabi$ maps $\winwz$ into $\winwz$.

Let $f \in \winwz$. The Walnut representation of $\opgabi$ in Theorem
\ref{th_walnut_rep} gives
$\opgabi(f)=\beta_i^{-d} \sum_j G^i_j T_{j/{\beta_i}} f$ with absolute
convergence in the norm of $\winw$. Hence it suffices to observe that each
of the functions $G^i_j$ is continuous. According to Theorem
\ref{th_walnut_rep} these are given by
\begin{align*}
G^i_j(x) := \sum_{k \in \Zdst}
\overline{g^i(x- j/{\beta_i}-\alpha_i k)} g^i(x-\alpha_i k).
\end{align*}
Since the function $g^i$ is continuous it suffices to note that in the
last series the convergence is locally uniform. This is an easy consequence
of the fact that $\norm{g^i}_\winw <+\infty$.
\end{proof}

\section*{Acknowledgements}

R.~Balan was supported by NSF grant DMS-0807896. I.~A.~Krishtal was partially
supported by the NSF grant DMS-0908239. He is grateful to A.~Baskakov and
K.~Gr\"ochenig for stimulating discussions.  J.~G.~Christensen and 
K.~A.~Okoudjou  were supported by ONR grants: N000140910324 $\&$ N000140910144.
K.~A.~Okoudjou was also partially supported by  a RASA from the Graduate School
of UMCP and by the Alexander von Humboldt foundation. He would also like to
express his gratitude to the Institute for Mathematics at the 
University of Osnabr\"uck for its hospitality while part of this work was
completed. J.~L.~Romero's research was supported by a Fulbright research grant,
during which he visited the Norbert Wiener Center for Harmonic Analysis
and Applications at the University of Maryland. He thanks the NWC for its kind hospitality and the Fulbright foundation
for its support. J.~L.~Romero also ackowledges support from the Austrian Science Fund (FWF):[M1586-N25] and
[P22746-N13].


\begin{thebibliography}{10}

\bibitem{abk08}
A.~Aldroubi,  A.~Baskakov, and I.~ Krishtal,  
\newblock {S}lanted matrices, {B}anach frames, and sampling.
\newblock{\em J. Funct. Anal.},  {\bf 255}(7): 1667--1691, 2008.

\bibitem{ba99}
R.~{B}alan, 
\newblock  {D}ensity and redundancy of the noncoherent {W}eyl-{H}eisenberg superframes. 
\newblock In {\em {T}he functional and harmonic analysis of wavelets and frames (San Antonio, TX, 1999), 29--41, Contemp.\ Math.,} {\bf 247}, Amer. Math. Soc., Providence, RI, 1999.

\bibitem{badau03}
R.~{B}alan and I.~{D}aubechies, 
\newblock {O}ptimal stochastic encoding and approximation schemes using {W}eyl-{H}eisenberg sets.
\newblock In {\em  {A}dvances in {G}abor analysis,} 259--320, {\em {A}ppl. {N}umer. {H}armon. {A}nal.,} {B}irkh\"auser {B}oston, {B}oston, {MA}, 2003.


\bibitem{ba10-1}
R.~{B}alan and I.~{K}rishtal.
\newblock {A}n almost periodic noncommutative {W}iener's {L}emma.
\newblock {\em {J}. {M}ath. {A}nal. {A}ppl.}, {\bf 370}(2): 339--349, 2010.

\bibitem{ba06-1}
R.~{B}alan.
\newblock The noncommutative {W}iener lemma, linear independence, and spectral
properties of the algebra of time-frequency shift operators.
\newblock {\em {T}rans. {A}mer. {M}ath. {S}oc.}, {\bf 360}(7): 3921--3941, 2008.

\bibitem{bacahela06}
R.~{B}alan, P.~G. {C}asazza, C.~{H}eil, and Z.~{L}andau.
\newblock {D}ensity, overcompleteness, and localization of frames {I}:
  {T}heory.
\newblock {\em {J}. {F}ourier {A}nal. {A}ppl.}, {\bf 12}(2): 105--143, 2006.

\bibitem{bacahela06-1}
R.~{B}alan, P.~G. {C}asazza, C.~{H}eil, and Z.~{L}andau.
\newblock {D}ensity, overcompleteness, and localization of frames {I}{I}:
  {G}abor frames.
\newblock {\em {J}. {F}ourier {A}nal. {A}ppl.}, {\bf 12}(3): 307--344, 2006.

\bibitem{Bas92}
A.~G. Baskakov. \newblock Abstract harmonic analysis and asymptotic estimates for elements of inverse matrices. \newblock {\em Math. Notes} \textbf{52}(1-2): 746--771, 1992.

\bibitem{Bas97vuz}
A.~G. Baskakov. \newblock Spectral analysis of difference operators with strongly continuous coefficients. \newblock {\em Russian Math. (Iz. VUZ)} \textbf{41}(10): 1--10, 1997.

\bibitem{BK05}
A.~G. Baskakov and I.~A. Krishtal. \newblock {H}armonic analysis of causal operators
  and their spectral properties, \newblock {\em Izv.~Ross.~Akad.~Nauk Ser.~Mat.} \textbf{69}
  (3):3--54, 2005.


\bibitem{BK11}
A.~{B}askakov and I.~{K}rishtal, \newblock{Memory estimation of inverse
  operators}. \newblock arXiv:1103.2748v1, 2011.



\bibitem{co90}
J.~B. {C}onway.
\newblock {\em {A} {C}ourse in {F}unctional {A}nalysis. 2nd ed.}
\newblock {S}pringer, {N}ew {Y}ork, 1990.

\bibitem{fe83}
H.~G. {F}eichtinger.
\newblock {B}anach convolution algebras of {W}iener type.
\newblock In {\em {P}roc. {C}onf. on {F}unctions, {S}eries, {O}perators
 ({B}udapest 1980}),  509--524,  {\em {C}olloq. {M}ath. {S}oc. {J}anos
  {B}olyai}, 
  {\bf 35}, {N}orth-{H}olland, {A}msterdam,  
  1983.

\bibitem{fe92-3}
H.~G. {F}eichtinger.
\newblock {W}iener amalgams over {E}uclidean spaces and some of their
  applications.
 \newblock In {\em Function spaces ({E}dwardsville, {IL}, 1990)}, 123--137,
 {\em Lecture Notes in Pure and Appl. Math.,}
\newblock {\bf 136}, {Dekker, New York, NY,} 1992.

\bibitem{fegr97}
H.~G. {F}eichtinger and K.~{G}r{\"o}chenig.
\newblock {G}abor frames and time-frequency analysis of distributions.
\newblock {\em {J}. {F}unct. {A}nal.}, {\bf 146}(2): 464--495, 1997.

\bibitem{fewe07}
H.~G. {F}eichtinger and F.~{W}eisz.
\newblock {G}abor analysis on {W}iener amalgams.
\newblock {\em {S}ampl. {T}heory {S}ignal {I}mage {P}rocess}, pages 129--150,
  {M}ay 2007.

\bibitem{fost85}
J.~J.~F. {F}ournier and J.~{S}tewart.
\newblock {A}malgams of ${L}^p$ and $\ell^q$.
\newblock {\em {B}ull. {A}m. {M}ath. {S}oc., {N}ew {S}er.}, {\bf 13}: 1--21, 1985.

\bibitem{grle01}
L.~{G}rafakos and C.~{L}ennard.
\newblock {C}haracterization of $L^p (R^n)$ using {G}abor frames.
\newblock {\em {J}. {F}ourier {A}nal. {A}ppl.}, {\bf 7}(2): 101--126, 2001.


\bibitem{Gr01}
\newblock K.~{G}r\"ochenig.
\newblock {\em {F}oundations of {T}ime-{F}requency {A}nalysis.}
\newblock {B}irkh\"auser,
Boston, 2001.

\bibitem{gr04}
K.~{G}r{\"o}chenig.
\newblock {L}ocalization of frames, {B}anach frames, and the invertibility of
  the frame operator.
\newblock {\em {J}. {F}ourier {A}nal. {A}ppl.}, {\bf 10}(2): 105--132, 2004.

\bibitem{grhe01}
K.~{G}r{\"o}chenig and C.~{H}eil.
\newblock {G}abor meets {L}ittlewood-{P}aley: {G}abor expansions in $
  {L}^p({R}^ d)$.
\newblock {\em {S}tudia {M}ath.}, {\bf 146}(1): 15--33, 2001.

\bibitem{grheok02}
K.~{G}r{\"o}chenig, C.~{H}eil, and K.~{O}koudjou.
\newblock {G}abor analysis in weighted amalgam spaces.
\newblock {\em {S}ampl. {T}heory {S}ignal {I}mage {P}rocess.}, {\bf 1}(3): 225--259,
  2002.

\bibitem{gr06}
K.~{G}r{\"o}chenig and M.~{L}einert.
\newblock {S}ymmetry and inverse-closedness of matrix algebras and functional
calculus for infinite matrices.
\newblock {\em {T}rans. {Amer}. {Math}. {Soc}.}, {\bf 358}(6): 2695--2711, 2006.

\bibitem{grle04}
K.~{G}r{\"o}chenig and M.~{L}einert.
\newblock {W}iener's lemma for twisted convolution and {G}abor frames.
\newblock {\em {J}. {A}mer. {M}ath. {S}oc.}, {\bf17}: 1--18, 2004.

\bibitem{hala00}
D.~{H}an D.~R.~{L}arson.
\newblock {\em {F}rames, bases and group representations.}
\newblock {M}em. {A}mer. {M}ath. {S}oc., {\bf 147} (2000), no.\ 697,

\bibitem{ja95}
A.~J. E.~M. {J}anssen.
\newblock {D}uality and biorthogonality for {W}eyl-{H}eisenberg frames.
\newblock {\em {J}. {F}ourier {A}nal. {A}ppl.}, {\bf 1}(4): 403--436, 1995.


\bibitem{K11}
I.~{K}rishtal. 
\newblock {W}iener's lemma and memory localization.
\newblock  {\em {J}. {F}ourier {A}nal. {A}ppl.}, \textbf{17} (4): 674--690, 2011. Doi: 10.1007/s00041-010-9152-3.


\bibitem{krok08}
I.~{K}rishtal and K.~{O}koudjou.
\newblock {I}nvertibility of the {G}abor frame operator on the {W}iener
amalgam
  space.
\newblock {\em {J}. {A}pprox. {T}heory}, {\bf 153}(2): 212--214,  2008.


\bibitem{K99}
V.~G. Kurbatov. \emph{Functional-differential operators and equations}.
  Mathematics and its Applications, vol. 473, Kluwer Academic Publishers,
  Dordrecht, 1999.

\bibitem{Lax02} P.~Lax. \emph{Functional Analysis}, John Wiley \& Sons,
 2002.

\bibitem{st01-1}
T.~{S}trohmer.
\newblock {A}pproximation of dual {G}abor frames, window decay, and wireless
communications.
\newblock {\em Appl. Comput. Harmon. Anal.}, 11(2):243--262, 2001. 

\bibitem{wa92}
D.~F. {W}alnut.
\newblock {C}ontinuity properties of the {G}abor frame operator.
\newblock {\em {J}. {M}ath. {A}nal. {A}ppl.}, {\bf 165}(2): 479--504, 1992.

\bibitem{we09-1}
F.~{W}eisz.
\newblock {P}ointwise summability of {G}abor expansions.
\newblock {\em {J}. {F}ourier {A}nal. {A}ppl.}, {\bf 15}(4): 463--487, 2009.

\end{thebibliography}
\end{document}